\documentclass[12pt]{article}

\usepackage{amsmath}
\usepackage{tikz}
\usepackage{amsfonts}
\usepackage{amssymb}
\usepackage{amsthm}
\usepackage{graphicx}

\newtheorem{theorem}{Theorem}
\newtheorem{lemma}[theorem]{Lemma}

\newtheorem{remark}[theorem]{Remark}
\newtheorem{definition}[theorem]{Definition}

\newcommand{\pr}[1]{\mathbb{P}\!\left(#1\right)}
\newcommand{\prcond}[2]{\mathbb{P}\!\left(#1\;\middle\vert\;#2\right)}

\newenvironment{claim}[1]{\par\noindent\underline{Claim:}\space#1}{}

\title{Random walks on hyperplane arrangements and stopping times}

\date{}

\author{Evita Nestoridi}

\begin{document}
\maketitle
\begin{abstract}
Consider a real hyperplane arrangement and let $\mathcal{C}$ denote the occurring chambers. Bidigare, Hanlon and Rockmore introduced a Markov chain on $\mathcal{C}$ which is a generalization of some card shuffling models used in computer science, biology and card games. This paper introduces strong stationary arguments for this Markov chain, which provide explicit bounds for the separation distance.
\end{abstract}

\section{Introduction}
Consider the following process on a finite, transitive graph: pick a vertex at random and flip a fair coin to determine whether to color this vertex and its neighbors red or blue. Now, consider the following card shuffling scheme: enumerate the subsets of $\{1,2,\ldots,n \}$ and assign weight $w_i$ to the $i^{th}$ subset. Pick a subset of $\{1,2,\ldots,n \}$ according to $w$ and move the cards indicated by that set to the top keeping their relative order.  This is a generalization of the riffle shuffles, called  the pop shuffles model. It turns out that these two processes are quite similar: they both are  Markov processes  on the chambers of some hyperplane arrangement.

A very special case of the second example  is the Tsetlin library or (weighted) random to top card shuffling: consider a collection of books (or cards), labeled $1$ through $n$. Pick a book $i$ with probability $w_i$ and move it to the front. This is a very well studied Markov chain mainly because of its use in dynamic file maintenance and cache maintenance (\cite{Do}, \cite{FHo}, \cite{Phatarfod}). The eigenvalues of this process were discovered independently by Donnelly \cite{Do}, Kapoor and Reingold \cite{KR}, and Phatarfod \cite{Phatarfod} .

Most of the processes on graphs of this type are viewed as Markov chains on the chambers of the Boolean arrangement. The card shuffling schemes mentioned above are treated as  Markov chains on the chambers of the braid arrangement.  Examples of card shuffling, hypercube walks and coloring processes are studied thoroughly in Sections \ref{braid} and \ref{examples}.

The unifying picture is the following: let $\mathcal{A}$ be a finite collection of affine hyperplanes in $V= \mathbb{R}^n$ which is called a hyperplane arrangement.  These hyperplanes cut $V$ in finitely many connected, open components that are called chambers. The chambers are finite intersections of half-spaces and therefore they have faces.

To define the chambers and the faces of a hyperplane arrangement, notice that a hyperplane cuts the space into two half spaces, call one of them positive and the other one negative. A chamber can be specified by keeping track for every hyperplane of whether it is on the positive or negative half space of the hyperlane.  Let $m$ be the number of hyperplanes in  $\mathcal{A}$. A chamber can be expressed as a vector with $m$ coordinates, each one of them is either $+$ or $-$. A face can also be viewed as vector with $m$ coordinates but this time the coordinates can also be zero, if the face lies on the hyperplane.

 Let $\mathcal{F}$ be the set of all faces and $\mathcal{C}$ be the set of all chambers. The following hyperplane arrangement in $\mathbb{R}^2$ contains $7$ chambers and $19$ faces (chambers, edges and points):

\centerline{
\begin{tikzpicture}
\draw (0,1) -- (4,-2) node[anchor=north west] {$F$}; 
\draw (3,2.5) -- (3,-2.5);
\draw (0,-2) -- (4,2) ;
\draw (3.5,2.3)  node{$C_1$} ;
\draw (4,0.5)  node{$C_2$} ;
\draw (1.5,1)  node{$C_0$};
\draw (3.5,-2)  node{$C_3$} ;
\draw (2,-2)  node{$C_4$} ;
\draw (0.5,-0.5)  node{$C_5$} ;
\draw (2.5,-0.5)  node{$C_6$} ;
\end{tikzpicture}}
$$\mbox{Figure }1$$

There is a notion of product between a face $F$ and a chamber $C$. The result will be the unique chamber which is the nearest to $C$ (in the sense of crossing the fewest number of hyperplanes) and has $F$ as a face, in other words faces act on chambers in the above way. The product $FC$ is called the projection of $C$ on $F$. For example, in figure $1$ the product of $C_0$ with $F$ is $C_2$. The product of faces is defined more carefully in section \ref{hyp} where it is shown to have the following associative property:
$$F(GC)= (FG)C$$
for all $F,G \in \mathcal{F}$ and $C \in \mathcal{C}$. The rigorous algebraic definition of the product is introduced in section \ref{hyp}.

Bidigare, Hanlon and Rockmore (BHR) \cite{BHR} defined a random walk on $\mathcal{C}$ using the above action of  $\mathcal{F}$ on $\mathcal{C}$ and characterized its eigenvalues. Starting with a probability measure $w$ on $\mathcal{F}$, a step in the walk is the following: from $C \in \mathcal{C}$, choose $F$ according to $w$ and move to $FC$. Denote by $C^t$ the $t^{th}$ configuration of the walk, that is the chamber the walk is on after $t$ steps of running the process. Then,
$$C^t=F^{t} \ldots F^{2}F^{1}C_0$$
where $F^i$ denotes the face picked at time $i$.

Brown and Diaconis \cite{BD} proved that that the transition matrix $K$ of this Markov Chain is diagonalizable and they reproved the BHR result. They also found a necessary and sufficient condition on $w$ so that $K$ has a unique stationary distribution. This condition is that $w$ separates the hyperplanes of $\mathcal{A}$, namely for every $H \in \mathcal{A}$ there is a face $F \nsubseteq H$ such that $w(F)>0$. Under that assumption, they provide a stochastic description for the stationary measure $\pi$: sample without replacement from $w$ and apply these faces in inverse order to any starting chamber (this way the first chosen face is the last to be applied). 
In this paper, $w$ is assumed to be separating, so that there exists a notion of convergence to this unique distribution. Athanasiadis and Diaconis have a similar discussion in \cite{AthD}, but they use purely combinatorial methods as well as a coupling argument.

The approach of this paper is more probabilistic. It involves a strong stationary time argument. It thus gives stronger bounds that previous methods; bounds in separation distance, which is defined as:
$$s(t)= \max_{x_0 \in \mathcal{C}} \left( 1-  \min_{x \in \mathcal{C}} \frac{K_{x_0}^{*t}(x)}{\pi(x)} \right)$$
where $K_{x_0}^{*t}(x)$ denotes the probability of starting the process at $x_0$ and moving to $x$ after $t$ steps.
 To state the result consider the following definition:
\begin{definition}
Let $F,G$ be two faces and denote by $I_F=\{H \in \mathcal{A}: F \subset H\}$. Then $F$ and $G$ are called adjacent if $$I_F=I_G$$
\end{definition}
This way the space $\mathcal{F}$ is partitioned in blocks $B_i$ each one consisting only of faces adjacent to the $i^{th}$ hyperplane. Let $$w(B_j)= \sum_{F \in B_j} w(F)$$ then the first new result of this paper states

\begin{theorem}\label{hyperplane}
Let $\mathcal{A}$ be a hyperplane arrangement and $w$ the measure on $\mathcal{F}$. If $K$ is the transition matrix of the Markov Chain described above then
$$s(t) \leq \sum_j (1-w(B_j))^t$$
where the sum is taken over all blocks of positive weight.
\end{theorem}

In particular, for the Tsetlin library described above, Theorem \ref{hyperplane} says that
\begin{theorem}\label{wr}
$$ s(t)\leq \sum^n_{i=1}(1-w(i))^t$$
where $w(i)$ is the weight of the $i^{th} $ card.
\end{theorem}

Theorem \ref{wr} gives the correct answer for the mixing time if the weights are all equal to $1/n$, which is $n \log n +cn$. 
Yet in some cases, such as the riffle shuffles,  Theorem \ref{hyperplane} does not give such accurate answers. Section \ref{symmetry} examines a special case where the following symmetry condition is required: assume that a group $G$ acts on $V$ preserving the hyperplane arrangement $\mathcal{A}$ so that the action restricted on the chambers is transitive. If for $F,L \in \mathcal{F}$ there is $g\in G$ such that 
\begin{equation}\label{condition}
F=gL \mbox{ then we require that } w(F)=w(L)
\end{equation} 
In this case the result is
\begin{theorem}\label{sym.con}
Under the symmetry conditions,
$$s(t)\leq \sum^m_{i=1}\left(1 - \sum_{\substack{F \in \mathcal{F}\\ F \notin H_i}}w(F) \right)^t$$
\end{theorem}

A few important examples can be found in sections \ref{braid} and \ref{examples}. Section \ref{hyp} includes the setup for the strong stationary time, as well as the proof that it is a strong stationary time indeed. Finally, Section \ref{pf} gives the details of the proof of Theorem \ref{hyperplane}. 

\begin{remark}
This paper provides only the basic information around hyperplane arrangements that is needed for the setup of the problem and the proof of the results. The reader is encouraged to learn more about hyperplane arrangements by reading \cite{Stanley}.

\end{remark}

\section{Strong stationary times}
The main results of this paper are proven using strong stationary times. Diaconis and Aldous \cite{A-P} introduced the following definition;
\begin{definition}
Fix $x_0\in X$. A strong stationary time is a stopping time $\tau$ such that for every $A \subset X$ and $k\geq 0$ it holds that
$$\mathbb{P}_{x_0 }\left(X_k \in A \vert \tau \leq k\right)= \pi(A)$$
where $X_k$ is the state that the Markov Chain is at time $k$.
\end{definition}
Aldous and Diaconis \cite{A-P} proved the following theorem which is the main link between strong stationary times and separation distance:

\begin{lemma}\label{inequality}
If $\tau$ is a strong stationary time then for $t>0$,
$$s(t) \leq \pr{\tau > t}$$
\end{lemma}

\section{Braid Arrangement and Pop shuffles.}\label{braid}
Shuffling schemes can be viewed as a Markov chain on the chambers of the braid arrangement. As presented in detail later, the chamber's of this arrangement are indexed by permutations and the faces are indexed by ordered block partitions. The following scheme is an example of such a Markov chain:

Consider all ordered block partitions of $[n]= \{1,2,\ldots n \}$ and assign weights to them. The card shuffling suggests to pick an ordered block partition $A_1, A_2,\ldots A_m$ according to the weights and remove from the deck the cards indicated by $A_1$ and put them on the top, keeping their relative order fixed. Then put the cards indicated by $A_2$ and put them exactly below the $A_1$ cards, keeping their relative order fixed and so on. This card shuffling is known as the pop shuffle.

This  shuffling  scheme is a Markov Chain on the chambers of a specific hyperplane arrangement. In particular consider the hyperplanes 
\begin{equation}
x_i=x_j
\end{equation}
The following pictures represent the braid arrangement for $\mathbb{R}^2$ and  $\mathbb{R}^3$ respectively, where in $\mathbb{R}^3$ the lines drawn correspond to planes:

\centerline{
\begin{tikzpicture}
\draw (0,2) -- (4,-2) node[anchor=north west] {$x_1=x_2$}; 
\draw (2,2.5) -- (2,-2.5) node[anchor=north west] {$x_1=x_3$};
\draw (0,-2) -- (4,2)node[anchor=north west] {$x_3=x_2$};
\draw (4,0)  node{$x_1<x_2<x_3$} ;
\draw (3,2)  node{$\substack{x_1<\\ x_3<x_2}$} ;
\draw (1,2)  node{$\substack{x_3<\\ x_1<x_2}$} ;
\draw (0,0)  node{$x_3<x_2<x_1$};
\draw (1.5,-1.5)  node{$\substack{x_2<\\ x_3<x_1}$} ;
\draw (2.7,-1.5)  node{$\substack{x_2<\\ x_1<x_3}$} ;
\draw (-2,2) -- (-6,-2) node[anchor=north west] {$x_1=x_2$};
\draw (-3,-1)  node{$x_1>x_2$} ;
\draw (-6,0)  node{$x_2>x_1$};
\end{tikzpicture}}

Then the chambers are in one to one correspondence with $S_n$.
That is because in the interior of a chamber none of the coordinates are equal to each other and in fact the ordering of the coordinates is fixed. For example,  the chamber  that corresponds to $\sigma \in S_b$ is
\begin{equation}
x_{\sigma(n)} < x_{\sigma(n-1)}<\ldots x_{\sigma(1)}
\end{equation}

The faces are exactly the ordered partitions of $[n]$, meaning that some of the coordinates are equal, forming these way blocks that are ordered. For example,
\begin{equation*}
\{1,2,3\}\{4,5\} \{6,7,\ldots n\}
\end{equation*} 
correspond to 
\begin{align*}
 x_1=x_2=x_3, x_4=x_5 &, x_6=x_7=\ldots = x_n \\
  x_1< & x_4<x_6
\end{align*}
Therefore, for this Markov Chain Theorem \ref{hyperplane} says that
$$s(t) \leq \sum (1-w(B_j))^t$$
but if we add the symmetry conditions of $T_3$ then we might be able to get better bounds.

\paragraph{Weighted subsets Markov Chain.}
Label the cards of a deck with the numbers $1,2\ldots n$ from top to end. Let $S_i$ be any subset of $\{1,2, \ldots n \}$ and $w_i$ the weight assigned to $S_i$ for $i=1,2\ldots nL$, where some of the $w_i$ are allowed to be zero. The only condition on the $w_i'$s is that for every $i,j \in \{1,2,3,\ldots n \}$ there is a subset $A$ of $1,2\ldots n$ with positive weight such that $i \in A$, $j \notin A $. 

 The Markov chain of this section picks a subset $S_j$ with probability $w_j$ and  then look at the deck of cards: remove the cards whose assigned number is in $S_j$ and move them to the top of the deck keeping their previous relative order. The stationary measure is sampling without replacement according to the $w_j'$s and perform the sorting on the deck of cards starting from the last subset picked.

During the Markov Chain process, $i,j \in \{1,2,3,\ldots n \}$ have been separated if  at least once we have picked a subset of $\{1,2,3,\ldots n \}$ which contains only one of $i,j$. 

In this case, Theorem \ref{hyperplane} says the following:
$$s(t) \leq \sum (1-w_j)^t$$
which is proven again by the same strong stationary time argument. Let $T$ be the first time that all subsets of $\{1,2,\ldots n\}$ with positive weight have been picked. Notice that if time $T$ has occurred then all pairs $i,j$ have been separated.  This $T$ is known to be a coupling time due to work of Athanasiadis and Diaconis \cite{AthD}.
\begin{lemma}\label{st}
$T$ is a strong stationary time.
\end{lemma}
The proof of lemma \ref{st} is omitted since it is a special case of the proof presented in Section \ref{one}. For more details on this example, see section $4$B of \cite{AthD}.

\paragraph{Inverse Riffle Shuffles.}
Inverse riffle shuffles, as presented by Aldous and Diaconis \cite{A-P}, relies on marking some of the cards with zeros and the rest with ones and then moving the former ones on top, preserving their relative order. This corresponds to sampling among the two-block ordered partitions $\{c_1,c_2,\ldots, c_i\} \{[n]- \{c_1,c_2,\ldots, c_i\} \}$ with weights:
$$w(B)= \begin{cases} 1/2^{n-1}, & \mbox{ if }B=[n]  \\
1/2^n, & \mbox{ if } B= (s,[n] \setminus s) \mbox{ where } s\neq \emptyset, s\neq [n] \\
 0, & \mbox{if } \mbox{ otherwise} \end{cases}$$

Although Bayer and Diaconis \cite{BaD} prove that that the optimal upper bound for the total variation mixing time is $\frac{3}{2} \log_2 n + \theta$ , yet work done by Aldous and Diaconis \cite{A-P} and Assaf, Diaconis and Soundararajan \cite{ADS} proves that the separation distance mixing time is $2 \log_2 n + \theta$. Several other metrics have also been studied: \cite{ADS} have studied the $l^{\infty}$ norm as well, while Stark, Gannesh and O'Connell \cite{SGO} studied the Kullback-Leibler distance.

 As discussed in Athanasiadis and Diaconis in \cite{AthD} there is a generalization of this card shuffling, namely marking the cards with a number in $\{0,1,\ldots,a-1\}$ according to the multinomial distribution. Then move the ones marked with zeros on top, keeping their relative order fixed, and continue with the ones marked with $1$ etc. This is a generalization of a strong stationary argument of Aldous' and Diaconis' in \cite{A-P}, giving an upper bound for the general inverse riffle shuffle of the form $2 \frac{\log n}{\log a}$.

Inverse shuffle is obviously a special case of the braid arrangement Markov chain and the weights assigned to an ordered partition are determined according to the multinomial distribution. Obviously, the only ordered partitions of positive weights are the ones that have at most $a$ blocks. The bounds given by the strong stationary time in this paper are not great in the case of the riffle shuffles. For example in the case where $a=2$, Theorem \ref{hyp} says:
$$s(t) \leq \sum^n_{i=0}{n \choose i} \left( 1- \frac{1}{2^n}\right)^t=2^n \left(1-\frac{1}{2^n} \right)^t$$
which gives an exponential bound for the mixing time.

This  is fixed by Theorem\ref{sym.con} since for $t=2 \log_2 n + c $
$$s(t) \leq \sum^{\frac{n(n-1)}{2}}_{i=1}\left( 1- \frac{2^{n-1}}{2^n}\right)^t = \frac{n(n-1)}{2}\left(  \frac{1}{2}\right)^t\leq \frac{1}{2^{c+1}}$$

\paragraph{Random to Top-Tsetlin Library.}\label{wrtt}
Let $w(j)$ denote the weight assigned to the $j^{th}$ card such that $w(j)>0$ for all $j \in \{ 1,2,\ldots n\}$ and $\sum^n_{j=1} w(j)=1$. Consider the following Markov Chain on $S_n$: start from a state $x$ in $S_n$. With probability $w(j)$ remove card $j$ and place it on top.

The stationary distribution is the Luce model, which has stationary distribution described as sampling from an urn with $n$ balls without replacement, picking ball $j$ with probability $w(j)$.

The eigenvalues of this Markov chain are known due to Phatarfod \cite{Phatarfod}.  Brown and Diaconis \cite{BD}, Athanasiadis and Diaconis \cite{AthD} also present the eigenvalues of the Tsetlin Library as an example of a hyperplane walk. Brown \cite{Brown} has analyzed the $q-$analogue of the Tsetlin library. In this section, a strong stationary argument is given: 

\begin{lemma}
For the Tsetlin Library with weights $w(i)$, let $T$ be the first time we have touched all cards. Then $T$ is a strong stationary time.
\end{lemma}
Roughly, consider first of all the case where all cards have weights $1/n$. Then the first time a card is moved to the top of the deck, the top card is a random card. When a new card is moved to the top then the order between the first two cards is random. Inductively, if there are $i$ cards on the top part of the deck with random order,given than a new, random card is moved to top will result to having $i+1$ cards on the top part of the deck in random order. Note that if given that $i$ cards have been touched and then one of them is chosen randomly to be moved to the top of the deck then the order of the $i$ top cards is still random, even conditional on $i$ and the times of moving.

Now if card $c$ has each own weight $w(c)$ then the 
probability that this card is moved to the top during the 
first step is exactly $w(c)$. Let's see what happens when 
a new card is moved to the top. Then the probability that 
card $a$ is on top, followed by card $c$ which is in the 
second position give that exactly two cards have been 
touched is $\frac{w(a)w(c)}{(1-w(c))}$. Assume that the 
probability of having $c_1$ on top, $c_2$ on the second 
position,$\ldots$, $c_i$ on the $i^{th}$ position, given 
that $i$ cards have been moved, is  
\begin{equation}\label{order}
\frac{w(c_1)w(c_2) 
\ldots w(c_i)}{(1- w(c_2) \ldots -w(c_i))\ldots (1-w(c_i))}. 
\end{equation}
Then given that on the next step a new card is move to 
the top, the probability of having $c_0$ on top, $c_1$ on 
the second position,$\ldots$, $c_{i}$ on the ${i+1}^{th}$ 
position is $\frac{w(c_0)w(c_1)w(c_2) \ldots c_i}{(1- 
w(c_1) \ldots -w(c_i))\ldots (1-w(c_i))}$. While if an already touched card gets moved to the top (\ref{order}) changes accordingly.

This is presented more formally in the following proof:

\begin{proof}
Let's mark the cards that we place on top. Let $T_i$ denote the $i^{th}$ time we mark a new card. Then the claim is that
\begin{align}\label{ddd}
\begin{split}
& \prcond{X^t(1)= c_1, X^t(2)= c_2, \ldots, X^t(i)= c_i }{ \begin{matrix}
T_i=t \\
c_1,\ldots, c_i \\
\mbox{are the marked cards at time } t
\end{matrix} } \\
&= \frac{w(c_1)}{w(c_1) + w(c_2)+ \ldots w(c_i)} \frac{w(c_2)}{ w(c_2)+ \ldots w(c_i)}\ldots \frac{w(c_{i-1})}{w(c_{i-1})+ w(c_i)}
\end{split}
\end{align}

To prove this use the following inductive argument. First of all, it is clear that $T_1=1$ and that
$$\prcond{X^1(1)= c_1 }{ \begin{matrix}
T_1=1 \\
c_1 
\mbox{ is the marked card at time } 1
\end{matrix} }=1$$
 and 
\begin{align*}
&\prcond{X^t(1)= c_1, X^t(2)= c_2 }{ \begin{matrix}
T_1=t \\
c_1 , c_2
\mbox{ are the marked cards at time } t
\end{matrix} }=\\
&\frac{w(c_1)}{w(c_1) +w(c_2)}
\end{align*}
Let's assume (\ref{ddd}) and take the inductive step
\begin{align*}
& \prcond{X^t(1)= c_1, X^t(2)= c_2, \ldots, X^t(i+1)= c_{i+1} }{ \begin{matrix}
T_{i+1}=t \\
c_1,\ldots, c_{i+1} \\
\mbox{are the marked cards at time } t
\end{matrix} } \\
&= \prcond{ X^t(2)= c_2, \ldots, X^t(i+1)= c_{i+1} }{ \begin{matrix}
T_{i+1}=t \\
c_1,\ldots, c_{i+1} \\
\mbox{are the marked cards at time } t
\end{matrix} }=\\
& \frac{w(c_1)}{w(c_1) + w(c_2)+ \ldots w(c_i) +w(c_{i+1})} \frac{w(c_2)}{ w(c_2)+ \ldots w(c_i) +w(c_{i+1})} \ldots \frac{w(c_i)}{w(c_{i})+w(c_{i+1})}
\end{align*}
The above almost finishes the proof. It remains to prove that if I move one of the marked cards then the distribution of the ordering of marked cards is the measure described by (\ref{ddd}).
I have to check that 
\begin{align*}
\begin{split}
& \prcond{X^t(1)= c_1, X^t(2)= c_2, \ldots, X^t(i)= c_i }{ \begin{matrix}
T_i=t-1 \\
c_1,\ldots, c_i \\
\mbox{are the marked cards at time } t
\end{matrix} } \\
&= \frac{w(c_1)}{w(c_1) + w(c_2)+ \ldots w(c_i)} \frac{w(c_2)}{ w(c_2)+ \ldots w(c_i)}\ldots \frac{w(c_{i-1})}{w(c_{i-1})+ w(c_i)}
\end{split}
\end{align*}
It is true because
\begin{align*}
& \prcond{X^t(1)= c_1, X^t(2)= c_2, \ldots, X^t(i)= c_i }{ \begin{matrix}
T_i=t-1, T_{i+1}>t \\
c_1,\ldots, c_i \\
\mbox{are the marked cards at time } t
\end{matrix} } =\\
& \prcond{X^t(1)= c_1, X^t(2)= c_2, \ldots, X^t(i)= c_i }{ \begin{matrix}
T_i=t-1 , T_{i+1}>t\\
c_1,\ldots, c_i \\
\mbox{are the marked cards at time } t\\
 X^{t-1}(1)=c_1,  X^{t-1}(2)= c_2, \\
 \ldots, X^{t-1}(i)= c_i
\end{matrix} }\\
& \prcond{X^{t-1}(1)=c_1,  X^{t-1}(2)= c_2,
 \ldots, X^{t-1}(i)= c_i}{\begin{matrix}
T_i=t-1 , T_{i+1}>t\\
c_1,\ldots, c_i \\
\mbox{are the marked cards at time } t
\end{matrix} }\\
& + \sum^i_{j=2} \prcond{X^t(1)= c_1, X^t(2)= c_2, \ldots, X^t(i)= c_i }{ \begin{matrix}
T_i=t-1 , T_{i+1}>t\\
c_1,\ldots, c_i \\
\mbox{are the marked cards at time } t\\
 X^{t-1}(j)=c_1,  X^{t-1}(1)= c_2, \\
 \ldots, X^{t-1}(i-1)= c_i
\end{matrix} } \\
& \prcond{X^{t-1}(j)=c_1,  X^{t-1}(1)= c_2,
 \ldots, X^{t-1}(i-1)= c_i}{\begin{matrix}
T_i=t-1, T_{i+1}>t \\
c_1,\ldots, c_i \\
\mbox{are the marked}\\
\mbox{ cards at time } t
\end{matrix} }=\\
&\frac{w(c_1)}{w(c_1) + w(c_2)+ \ldots w(c_i)} \\
& \Bigg(  \prcond{X^{t-1}(1)=c_1,  X^{t-1}(2)= c_2,
 \ldots, X^{t-1}(i)= c_i}{\begin{matrix}
T_i=t-1 , T_{i+1}>t\\
c_1,\ldots, c_i \\
\mbox{are the marked cards at time } t
\end{matrix} } + 
\end{align*}
\begin{align*}
&\sum^{i}_{j=2} \prcond{X^{t-1}(j)=c_1,  X^{t-1}(1)= c_2,
 \ldots, X^{t-1}(i-1)= c_i}{\begin{matrix}
T_i=t-1, T_{i+1}>t \\
c_1,\ldots, c_i \\
\mbox{are the marked}\\
\mbox{ cards at time } t
\end{matrix} } \Bigg)=\\
& \frac{w(c_1)}{w(c_1) + w(c_2)+ \ldots w(c_i)} \frac{w(c_2)}{ w(c_2)+ \ldots w(c_i)} \frac{w(c_{j-1})}{w(c_{j-1}) + w(c_{j})}
\end{align*}

\end{proof}

To prove Theorem \ref{wr} will simply use a union bound:
\begin{proof}
Let $A^t_i$ be the event that at $t$ steps I haven't touched card $c_i$. Then we have that
$$ P(T>t) \leq P(\cup^n_{i=1}A^t_i) \leq \sum^n_{i=1}(1-w_i)^t$$
\end{proof}

\paragraph{Random to top or bottom.}
Consider the card shuffling where a card is chosen at random and is moved to the top or the bottom with probability $1/2$. This is again a random walk on the chambers of the braid arrangement. The faces used are of the form $\{\{c\}, \{[n] \setminus \{c\}\}\}$ and $\{ \{[n] \setminus \{c\}\}, \{c\}\}$ each one having weight $1/2n$. Theorem  \ref{hyp} says that if $t= n \log n+cn$ then 
$$s(t)\leq n \left( 1- \frac{1}{n}\right)^t\leq e^{-c}$$
For the weighted version of the card shuffling let $w^+_c$ denote the weight of  $\{\{c\}, \{[n] \setminus \{c\}\}\}$ and  $w^-_c$ denote of $\{ \{[n] \setminus \{c\}, \{c\}\}\}$. Then theorem \ref{hyp} gives that
$$s(t)\leq \sum^n_{c=1}\left(1- w^-_c -w^+_c\right)^t$$

To finish off and get explicit bounds in the weighted cases depends a lot on the weights $w(i)$. For some sample calculations see the work of Diaconis \cite{cutoff}.

\section{The Boolean Arrangement.}\label{examples}
The Boolean arrangement consists simply of the hypeplanes $x_i=0$, $1\leq i \leq n$ in $\mathbb{R}^n$. Each chamber is specified by the sign of its coordinates, in other words they are the $2^n$
orthants in $\mathbb{R}^n$. The faces are in bijection with $\{-,0,+\}^n$. The projection $FC$ of a chamber $C$ on a face $F$ is a chamber who adopts all the signs non-zero coordinates of $F$ an the rest  of the coordinates have the signs of $C$.
\paragraph{Neighborhood walk on the hypercube.}
The chambers of the Boolean arrangement are as explained above in a bijection with  $\{-,+\}^n$, in other words each chamber corresponds to a vertex of the $n-$dimensional hypercube. If the only positive weighted faces are the $E^{\pm}_i$, whose $i^{th}$ coordinate is $\pm$ and the rest are zero, then the Markov Chain corresponds to the weighted nearest neighbor random walk on the hypercube, which corresponds to choosing a coordinate and switching it to $\pm$. Denote the weight of $E^{\pm}_i$ by $ w_i^{\pm}$. The transition matrix in this case is
$$K(x,x')= \begin{cases} \sum^n_{i=1} w_i^{x_i}, &  \mbox{ if } x=x'\\
w_i^{-x_i}, &\substack{\mbox{if } x \mbox{ is obtained from }x' \\  \mbox{ by switching the }i^{th} \mbox{ coordinate of x}}\\
0, & \mbox{ otherwise.}
\end{cases}$$

A strong stationary time in this case is the first time that all coordinates have been picked. Then
$$s(t) \leq \sum^n_{i=1} (i-w_i^+)^t+ \sum^n_{i=1} (i-w_i^-)^t
$$
 
In particular, for the case where $w_i^{\pm}= \frac{1}{2n}$ the upper bound for the separation time mixing time will be bounded by  $2n \log 2n+cn$, but theorem \ref{sym.con} improves the bound to $n \log n+cn$. It is straightforward to show a lower bound of the form $n \log n -cn $ for separation distance. This random walk is very well studied: Aldous \cite{Aldous_correct} and Diaconis and Shashahani \cite{PDMS} have proved the cut-off for the total variation distance mixing time at $\frac{n}{2} \log n + cn$ using Fourier analysis.

\paragraph{A non-local walk on the hypercube.}
Consider the following walk on the hypercube:  fix $k \geq 1$ and
pick $k$ coordinates at random and flip a fair coin for each one of them to determine whether to turn them into ones or zeros. In this case, Theorem \ref{sym.con} gives an upper bound of the form $\frac{n}{k} \log n + c\frac{n}{k}$, since:
$$s(t)\leq \sum^n_{i=1} \left( 1- \frac{{n-1 \choose k-1}}{ {n \choose k}}\right)^t= \sum^n_{i=1} \left( 1- \frac{k}{n}\right)^t$$

\paragraph{A walk on a finite, trasitive graph.}
As described at the beginning of this paper, a special case of a hyperplane arrangement walk could be the following process on a finite, transitive graph: pick a vertex at random and color it and its neighbors all red or blue with probability $1/2$. Then the strong stationary time suggests to stop once at least one representative of each neighborhood  has been picked. If $S$ is a minimum vertex cover then theorem \ref{sym.con} says that
$$s(t) \leq |S| \left( 1- \frac{1}{n}\right)^l $$
\begin{remark}
The described process can be considered for any type of graph and the coupling bound of Athanasiadis and Diaconis works. Yet to pass to separation distance and  use Theorem \ref{sym.con} transitivity is needed because of the symmetry conditions.  
\end{remark}

\section{Preliminaries}\label{hyp}
In the following sections, three strong stationary time arguments will be analyzed for the general hyperplane arrangement Markov chain. The easiest one to state is  the first time all positively weighted faces have been picked. Call this time $T_1$. 
The second strong stationary time $T_2$ is very similar to $T_1$, so more details can be found in section \ref{second}. The third strong stationary time, which will be called $T_3$ is the first time that the $F_{i_1} F_{i_2}\ldots F_{i_l}$ is a chamber, given that $F_{i_j}$ is the face picked at time $j$. $T_3$ gives much better bounds than $T_1$ and $T_2$.

This primary section proves a useful geometric lemma that will simplify the explanation that the above stopping times are indeed strong stationary times. Moreover, the following definition is the key fact behind that lemma. By definition, any face $F$ can be written in the following form:
$$F= \cap_{i \in I} H^{\sigma_i(F)}_i$$
where $\sigma_i(F) \in \{+,-,0 \}$, $H^+_i$ corresponds to the right open half-space determined by $H_i$ (and respectively $H^-_i$ for the left one) and $H^0_i= H_i$. Notice that if $\sigma_i(F)\neq 0$ for all $i$ if and only if $F$ is a chamber. The faces form a semigroup under the following product:
\begin{definition}
If $F,G$ are two faces then
$$FG= \cap_{i \in I} H^{\sigma_i(FG)}_i$$
where 
$$\sigma_i(FG)=\begin{cases} \sigma_i(F), & \mbox{if } \sigma_i(F)\neq 0 \\ 
 \\ \sigma_i(G), &  \mbox{ otherwise} \end{cases}$$
\end{definition}
It turns out that multiplication of faces satisfies both the ``idempotence" the ``deletion property", that is if $F$ and $G$ are two faces then
\begin{equation}\label{deletion}
F \cdot F=F \mbox{ and } FGF=FG
\end{equation} 
A property of this type has already appeared in special types of semigroups called left-regular bands. Brown \cite{Brown} has also used this property to find the eigenvalues of a similar Markov chain on semigroups. The "deletion property" leads to the  following lemma:
\begin{lemma}
Let $F_{i_j} \in \{ F \in \mathcal{F} : w(F)>0 \}$, then 
\begin{equation}\label{product}
F_{i_a}F_{i_k}F_{i_{k-1}} \ldots F_{i_{a+1}}F_{i_a}F_{i_{a-1}} \ldots F_{i_1}= F_{i_a}F_{i_k}F_{i_{k-1}} \ldots F_{i_{a+1}}F_{i_{a-1}} \ldots F_{i_1}
\end{equation}
In other words, if the left term of a product of faces has appeared in the product earlier, then it can be omitted from all the positions but the most left one and the product will remain the same. Also,
let $T_i$ denote the $i^{th}$ time a new face is picked then 
\begin{align*}
&\prcond{C^t= C}{\begin{matrix}
&T_i<t< T_{i+1} \\
& \mbox{the marked faces are}\\
& F_1,F_2,\ldots, F_i \\
&F_j \mbox{ was picked at time }t\\
& \mbox{and it is the }l+1 \mbox{ time it has been picked}\end{matrix}}=
\end{align*}
\begin{align}\label{prob}
& \prcond{C^{t-l}= C}{\begin{matrix}
&T_i=t-l \\
& \mbox{the marked faces are}\\
& F_1,F_2,\ldots, F_i \\
&F_j \mbox{ was picked at time }t
\end{matrix}}
\end{align}
for all $t$ that the condition could be applied to.
\end{lemma}
\begin{proof}
To prove that equation (\ref{product}) holds it suffices to check the deletion property described by equation (\ref{deletion}). This is easy because
\begin{equation*}
\sigma_i(F \cdot F)= \sigma_i(F) \mbox{ and } \sigma(FGF)= \sigma (FG)
\end{equation*}
Equation \ref{prob} holds because if 
$$C^t= F_jp^{t-1}(F_1,F_2,\ldots F_i) C_0$$
where $p^{t-1}(F_1,F_2,\ldots F_i)$ is a product of $F_1,F_2,\ldots F_i$ of length $t-1$ where $F_j$ appears exactly $l $ times and non of these terms is omitted and if 
$$C^{t-l}= F_jp^{t-1}(F_1,F_2,\ldots F_{j-1},\emptyset,F_{j+1} \ldots F_i) C_0$$
then 
$$C^t= C^{t-l}$$
and thus 
\begin{align*}
&\prcond{C^t= C}{\begin{matrix}
&T_i<t< T_{i+1} \\
& \mbox{the marked faces are}\\
& F_1,F_2,\ldots, F_i \\
&F_j \mbox{ was picked at time }t\\
& \mbox{and it is the }l+1 \mbox{ time it has been picked}\end{matrix}}=\\
& \frac{1}{{t-1 \choose l}}\sum_{1\leq m_1<m_2< \ldots <m_l<t } \prcond{C^t= C}{\begin{matrix}
&T_i<t<T_{i+1}\\
& \mbox{the marked faces are}\\
& F_1,F_2,\ldots, F_i \\
&F_j \mbox{ was picked at time }t\\
& \mbox{and it is the }l+1\\
& \mbox{ time it has been picked,}\\
&m_1,m_2,\ldots m_l \\
&\mbox{ are the moments that we picked } F_j \end{matrix}} =\\
& \frac{1}{{t-1 \choose l}}\sum_{1\leq m_1<m_2< \ldots <m_l<t } \prcond{C^{t-l}= C}{\begin{matrix}
&T_i=t-l \\
& \mbox{the marked faces are}\\
& F_1,F_2,\ldots, F_i \\
&F_j \mbox{ was picked at time }t
\end{matrix}}=
\end{align*}
\begin{align*}
& \prcond{C^{t-l}= C}{\begin{matrix}
&T_i=t-l \\
& \mbox{the marked faces are}\\
& F_1,F_2,\ldots, F_i \\
&F_j \mbox{ was picked at time }t
\end{matrix}}
\end{align*}

\end{proof}

\section{A first strong stationary argument}\label{one}
Let $\mathcal{A}$ be a hyperplane arrangement with face weights $w_F$, $F \in \mathcal{F}$. Assume that $w_F$ are separating.
Let $T_1$ be  the first time all positively weighted faces have been picked.
\begin{lemma}\label{T1}
$T_1$ is a strong stationary time for the Markov chain on $\mathcal{C}$.
\end{lemma}

Roughly, if face $F_1$ has each own weight $w(F_1)$ then the 
probability that $F_1$ is picked during the 
first step is exactly $w(F_1)$. Let's see what happens when 
a new face is picked. Then the probability that 
face $F_2$ is picked, after the last time $F_1$ was picked, given that exactly two faces have been 
touched is $\frac{w(F_2)w(F)_1}{(1-w(F_1))}$. Assume that the 
probability of having picked $F_1$ last, $F_2$ before that,$\ldots$, $F_i$ was the first face to be picked, given 
that $i$ faces have been picked, is  
\begin{equation}\label{order2}
\frac{w(F_1)w(F_2) 
\ldots w(F_i)}{(1- w(F_2) \ldots -w(F_i))\ldots (1-w(F_i))}. 
\end{equation}
Then given that on the next step a new face is picked, the probability of having picked $F_0$ last, $F_1$ before that,$\ldots$, $F_{i}$ firstly  is $\frac{w(F_0)w(F_1)w(F_2) \ldots w(F_i)}{(1- 
w(F_1) \ldots -w(F_i))\ldots (1-w(F_i))}$. While if an already touched face gets picked (\ref{order2}) changes accordingly.

This is presented more formally in the following proof:

\begin{proof}
The proof of the lemma is based on induction.  Every time a new face is picked it will be marked. 

\begin{claim}
 Let $T_i$ denote the $i^{th}$ time a new face is marked.
Then for any $F_1,F_2, \ldots F_{i} \in \mathbb{F}$ the following is true for all $t$ that make the condition possible:
 \begin{align*}
&\prcond{C^t= C}{\begin{matrix}
&T_{i}=t \\
& \mbox{the marked faces are}\\
& F_1,F_2,\ldots, F_{i} 
\end{matrix}} =\\
& \prcond{C^t= C}{\begin{matrix}
&T_{i-1}<t< T_i \\
& \mbox{the marked faces are}\\
& F_1,F_2,\ldots, F_{i} 
\end{matrix}} \\
& = w_{F_1,\ldots F_{i}}(C)
\end{align*} 
 where $ w_{F_1,\ldots F_{i}}(C)$ is sampling without replacement from $\{F_1,F_2\ldots F_i \}$ and applying the faces picked to $C_0$  in the reserve order, keeping track only of the products that will end up in $C$.
\end{claim}

It is clear that if the claim is true then $T_1$ is a strong stationary time.
First of all,
$$\prcond{C^t= C}{ \mbox{only set } F \mbox{ is marked at time } t}= \begin{cases} 1, & \mbox{if } F C_0=C
 \\ 0, & \mbox{ otherwise} \end{cases}$$
In order to do induction on $i$, assume that for any $F_1,F_2, \ldots F_{i-1} \in \mathbb{F}$ the following is true for all $t$ that make the condition possible:
 \begin{align*}
&\prcond{C^t= C}{\begin{matrix}
&T_{i-1}=t \\
& \mbox{the marked faces are}\\
& F_1,F_2,\ldots, F_{i-1} 
\end{matrix}} =\\
& \prcond{C^t= C}{\begin{matrix}
&T_{i-1}<t< T_i \\
& \mbox{the marked faces are}\\
& F_1,F_2,\ldots, F_{i-1} 
\end{matrix}} \\
&= w_{F_1,\ldots F_{i-1}}(C)
\end{align*}
 where $ w_{F_1,\ldots F_{i-1}}(C)$ is sampling without 
  replacement from $\{F_1,F_2\ldots F_{i-1} \}$ and applying the faces picked to $C_0$  in the reserve order, keeping track only of the products that will end up in $C$.
 
 Then 
\begin{align*}
&\prcond{C^t= C}{\begin{matrix}
&T_i=t \\
& \mbox{the marked faces are}\\
& F_1,F_2,\ldots, F_i 
\end{matrix}}= \\
& \sum^i_{j=1} \prcond{C^t= C}{\begin{matrix}
&T_i=t \\
& \mbox{the marked faces are}\\
& F_1,F_2,\ldots, F_i \\
& \mbox{face } F_j \mbox{ was chosen at time } t
\end{matrix}}\\
&\prcond{\mbox{face } F_j \mbox{ was chosen at time } t}{\begin{matrix}
&T_i=t \\
& \mbox{the marked faces are}\\
& F_1,F_2,\ldots, F_i 
\end{matrix}}=
\end{align*}
\begin{align}\label{eq}
& \sum^i_{j=1} \frac{w(F_j)}{w(F_1)+w(F_2)+ \ldots + w(F_i)} \prcond{C^t= C}{\begin{matrix}
&T_{i-1}<t \\
& \mbox{the marked faces are}\\
& F_1,F_2,\ldots, F_i \\
& \mbox{face } F_j \mbox{ was chosen at time } t
\end{matrix}}
\end{align}
 The induction step will help to prove that (\ref{eq}) is equal to $ w_{F_1,\ldots F_{i}}(C)$. The previous step configuration $C^{t-1}$ will run through all possible chambers $\overline{C}$ and therefore
\begin{align*}
&\prcond{C^t= C}{\begin{matrix}
&T_{i-1}<t \\
& \mbox{the marked faces are}\\
& F_1,F_2,\ldots, F_i \\
& \mbox{face } F_j \mbox{ was chosen at time } t
\end{matrix}} =\\
&\sum_{ \overline{C}: \mbox{ }F_j\overline{C}=C}  \prcond{C^{t-1}= \overline{C}}{\begin{matrix}
&T_{i-1}<t \\ & \mbox{the marked faces at time }t-1 \mbox{ are}\\
& F_1,F_2,\ldots, F_{i-1} \\
\end{matrix}}= \\
& \sum_{\overline{C}: \mbox{ }F_j\overline{C}=C}  w_{F_1,\ldots F_{j-1}, F_{j+1} \ldots F_{i}}(\overline{C})
\end{align*}
Thus 

\begin{align*}
&\prcond{C^t= C}{\begin{matrix}
&T_i=t \\
& \mbox{the marked faces are}\\
& F_1,F_2,\ldots, F_i 
\end{matrix}}=\\
& \sum^i_{j=1} \sum_{\overline{C}: \mbox{ }F_j\overline{C}=C} \frac{w(F_j)}{w(F_1)+w(F_2)+ \ldots + w(F_i)}w_{F_1,\ldots F_{j-1}, F_{j+1} \ldots F_{i}}(\overline{C})=\\
&w_{F_1, \ldots F_{i}}(C)
\end{align*}

 To complete the proof we need to also prove that the following measure also coincides with $w_{F_1, \ldots F_{i}}(C)$.
\begin{align*}
&\prcond{C^t= C}{\begin{matrix}
&T_i<t<T_{i+1}\\
& \mbox{the marked faces are}\\
& F_1,F_2,\ldots, F_i 
\end{matrix}}=\\
& \sum^i_{j=1}   \frac{w(F_j)}{w(F_1)+w(F_2)+ \ldots + w(F_i)} 
 \prcond{C^t= C}{\begin{matrix}
&T_i<t<T{i+1}\\
& \mbox{the marked faces are}\\
& F_1,F_2,\ldots, F_i \\
&F_j \mbox{ was picked at time }t
\end{matrix}}
\end{align*}
But equation (\ref{prob}) says that 
\begin{align*}
&\prcond{C^t= C}{\begin{matrix}
&T_i<t< T_{i+1}\\
& \mbox{the marked faces are}\\
& F_1,F_2,\ldots, F_i \\
&F_j \mbox{ was picked at time }t
\end{matrix}}=\\
& \sum_{l=0}^{t-i} \prcond{C^t= C}{\begin{matrix}
&T_i<t< T_{i+1}\\
& \mbox{the marked faces are}\\
& F_1,F_2,\ldots, F_i \\
&F_j \mbox{ was picked at time }t\\
& F_j \mbox{ has been picked } l+1 \mbox{ times}
\end{matrix}}\\
\end{align*}
\begin{align*}
& \prcond{ F_j \mbox{ has been picked } l+1 \mbox{ times}}{\begin{matrix}
&T_i<t< T_{i+1}\\
& \mbox{the marked faces are}\\
& F_1,F_2,\ldots, F_i \\
&F_j \mbox{ was picked at time }t
\end{matrix}}=\\
\end{align*}
\begin{align*}
& \sum_{l=0}^{t-i} \prcond{C^{t-l}= C}{\begin{matrix}
&T_i=t-l\\
& \mbox{the marked faces are}\\
& F_1,F_2,\ldots, F_i \\
&F_j \mbox{ was picked at time }t-l\\
\end{matrix}}\\
\end{align*}
\begin{align*}
&\prcond{ \substack{F_j \mbox{ has been picked } l+1 \mbox{ times}}}{\begin{matrix}
&T_i<t< T_{i+1}\\
& \mbox{the marked faces are}\\
& F_1,F_2,\ldots, F_i \\
&F_j \mbox{ was picked at time }t
\end{matrix}}=
\end{align*}
\begin{align*}
  &w_{F_1,\ldots F_{j-1}, F_{j+1} \ldots F_{i}}(C)  \sum_{l=0}^{t-i} \prcond{ \substack{F_j \mbox{ has been} \\ \mbox{picked } l+1 \mbox{ times}}}{\begin{matrix}
&T_i<t< T_{i+1}\\
& \mbox{the marked faces are}\\
& F_1,F_2,\ldots, F_i \\
&F_j \mbox{ was picked at time }t
\end{matrix}} =\\
  &w_{F_1,\ldots F_{j-1}, F_{j+1} \ldots F_{i}}(C) \\
\end{align*}
The last equality is because of induction. Therefore, indeed 

\begin{align*}
&\prcond{C^t= C}{\begin{matrix}
&T_i<t< T_{i+1}\\
& \mbox{the marked faces are}\\
& F_1,F_2,\ldots, F_i 
\end{matrix}}=
\end{align*}
\begin{align*}
& \sum^i_{j=1}   \frac{w(F_j)}{w(F_1)+w(F_2)+ \ldots + w(F_i)} 
 \prcond{C^t= C}{\begin{matrix}
&T_i<t< T_{i+1}\\
& \mbox{the marked faces are}\\
& F_1,F_2,\ldots, F_i \\
&F_j \mbox{ was picked at time }t
\end{matrix}}\\
&= w_{F_1, \ldots F_{i}}(C)
\end{align*}
and this leads to the proof of the claim.
\end{proof}

\section{An improved strong stationary time}\label{second}
This section introduces a strong stationary time for a special case of the Markov chain on hyperplane arrangements. Consider the following definition:
\begin{definition}
Let $F,G$ be two faces and denote by $I_F=\{H_i \in \mathcal{A}: \sigma_{i}(F) \neq 0 \}$. Then $F$ and $G$ are called adjacent if $$I_F=I_G$$
\end{definition}
For example, in the following picture $F_1,F_2$ and $F_3$ are adjacent since $F_1=(+,+,0)$, $F_2=(+,-,0)$ and $F_3=(-,-,0)$ if the coordinates are taken according to $h_1,h_2,h_3$ in order. More specifically
$I_{F_1}=I_{F_2}=I_{F_3}= \{ h_1,h_2\}$.

\begin{figure}[ht!]
\centering
\includegraphics[scale=0.4]{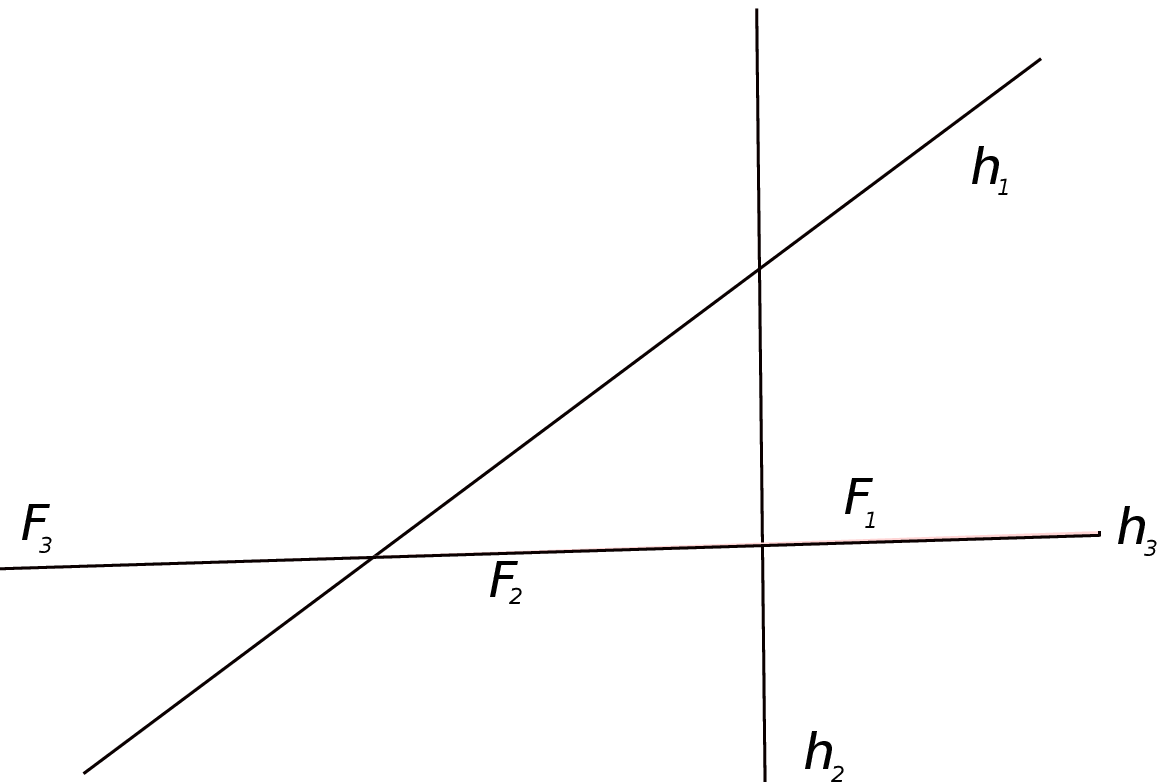}
\label{adj}
\end{figure}

Adjacency is an equivalence relationship and therefore there is a partition of the faces on blocks where each block contains adjacent faces. Each block is determined by the positions of the non-zero coordinates. The weight of the block $B_j$ is $$w(B_j)= \sum_{F \in B_j} w(F)$$ Let $T_2$ be the first time that at least one representative of each positive weighted block been picked. The proof of the fact that $T_2$ is a strong stationary time depends on a second version of the stationary measure, that is introduced in Brown's and Diaconis' section 3 \cite{BD}: 

\begin{remark}\label{descrip}
The stationary measure $\pi$ is the same as sampling faces $F_1, \ldots, F_l$ with replacement  until the outcome of $F_1 F_2 \ldots F_l$ is a chamber. 
\end{remark}
A similar description for the stationary measure is the following:
\begin{lemma}\label{char}
The stationary measure is the same as sampling with replacement from the faces and stop sampling the first time that at least one representative from each block has been picked. 
\end{lemma}
\begin{proof}
Let $q$ be the occurring measure when sampling without replacement from the faces, multiplying in the reverse order and stopping when all blocks have been represented. The goal is to prove that $q=\pi$. The idea is that unnecessary terms can be omitted.

Let $C \in \mathcal{C}$ and $F_{i_1},F_{i_2},\ldots,F_{i_j}, \ldots ,F_{i_l}$ be an ordering of the faces so that  $C=F_{i_1}F_{i_2}\ldots F_{i_j} \ldots F_{i_l}$ and $F_{i_1},F_{i_2},\ldots,F_{i_j}$ is such so that every block is represented but not every block is represented in $F_{i_1},F_{i_2},\ldots,F_{i_{j-1}}$. Then the following term appears $\pi(C) $ as a summand:
$$\sum_{\sigma \in S_{l-j}} \pr{\begin{matrix} 
&  F_{i_{1}} \mbox{was the first face picked} \\ & \ldots \\ & F_{ i_j} \mbox{was the } j^{th} \mbox{ face picked}\\
& \ldots \\
& F_{\sigma (i_l)} \mbox{was the last face picked} 
 \end{matrix} } =$$
$$\sum_{\sigma \in S_{l-j}} \pr{\begin{matrix} 
&  F_{\sigma (i_{j+1})} \mbox{was the }j+1 \mbox{ face picked} \\ & \ldots \\ & F_{\sigma (i_l)} \mbox{was the last face picked} \end{matrix} \middle\vert \begin{matrix} 
&  F_{i_{1}} \mbox{was the first face picked} \\ & \ldots \\ & F_{ i_j} \mbox{was the } j^{th} \mbox{ face picked} \end{matrix} } $$ $$\pr{ \begin{matrix} 
&  F_{i_{1}} \mbox{was the first face picked} \\ & \ldots \\ & F_{ i_j} \mbox{was the } j^{th} \mbox{ face picked} \end{matrix} } =  \pr{ \begin{matrix} 
&  F_{i_{1}} \mbox{was the first face picked} \\ & \ldots \\ & F_{ i_j} \mbox{was the } j^{th} \mbox{ face picked} \end{matrix} } $$
where $\pr{ \begin{matrix} 
&  F_{i_{1}} \mbox{was the first face picked} \\ & \ldots \\ & F_{ i_j} \mbox{was the } j^{th} \mbox{ face picked} \end{matrix} } $ appears as a summand in $q(C)$.  This justifies why $\pi=q$.
\end{proof}

\begin{lemma}
$T_2$ is a strong stationary time.
\end{lemma}

\begin{proof}
The goal is to prove that
$$\prcond{C^t=C}{T_2=t}= \pi(C) $$

Let  $ w_{F_1,\ldots F_{i}}(C)$ be sampling without replacement from $\{F_1,F_2\ldots F_i \}$
 and applying the faces picked to $C_0$  in the reserve order, keeping track only of the products that will end up in $C$, just like in the proof of Lemma \ref{T1}. Once more let $T_i$ be the first time a new face is picked. The following relation is proved in the same way as Lemma \ref{T1}:

\begin{align*}
&\prcond{C^t= C}{\begin{matrix}
&T_i<t<T_{i+1}\\
& \mbox{the marked faces are}\\
& F_1,F_2,\ldots, F_i 
\end{matrix}}=\prcond{C^t= C}{\begin{matrix}
&T_i=t\\
& \mbox{the marked faces are}\\
& F_1,F_2,\ldots, F_i 
\end{matrix}}\\
&=w_{F_1,\ldots F_{i}}(C)
\end{align*}

Remark \ref{char} makes it clear that summing over all $F_1,F_2,\ldots F_i$ are such so that for every block has a representative among the $F_j'$s then the following holds $$\pi(C)= \sum_{\substack{F_1,F_2,\ldots F_i } } w_{F_1,\ldots F_{i}}(C)$$

Therefore, if $F_1,F_2,\ldots F_i$ are such so that every block has a representative among them and if at every step the face picked is marked then
\begin{align*}
&\prcond{C^t= C}{\begin{matrix}
&T_2 \leq t\\
& \mbox{the marked faces are}\\
& F_1,F_2,\ldots, F_i 
\end{matrix}}\\
&=w_{F_1,\ldots F_{i}}(C)
\end{align*}

which completes the proof.
%
 
\end{proof}

\section{Proof of Theorem \ref{hyperplane}.}\label{pf}
\begin{proof}
The goal is to bound the right hand side of Lemma \ref{inequality}, which stated that 
$$s(t) \leq P(T_2>t)$$
$T_2$ gives better bounds (or constants) than $T_1$ therefore it is  preferable to use it over $T_2$.
The following union bound argument will help with bounding $P(T_2>t)$.More precisely, let $A^t_i$ be the event that at $t$ steps block $B_i$ hasn't been picked.  Also remember the notation
$w(B_j)= \sum_{F \in B_j} w(F)$.
Then we have that
$$ P(T_2>t) \leq P(\cup^m_{i=1}A^t_i) \leq \sum^m_{i=1}(1-w(B_i))^t$$
which finishes the proof of Theorem \ref{hyperplane}.
\end{proof}

\section{A more specialized, faster strong stationary time}\label{symmetry}
Let $T_3$ is the first time  that the product of faces picked is a chamber. According to Athanasiadis and Diaconis this is a coupling time \cite{AthD}. Assuming some symmetry conditions, this is also a strong stationary time. 

In this section, assume that a group $G$ acts on $V$ preserving the hyperplane arrangement $\mathcal{A}$ so that the action restricted on the chambers is transitive. Assume the symmetry conditions described by equation \ref{condition}. The first lemma concerns the stationary distribution:
\begin{lemma}
Under the symmetry conditions the stationary measure is the uniform measure on the chambers.
\end{lemma}
\begin{proof}
The fact that the stationary distribution $\pi$ is sampling without replacement until the product of the faces picked is a chamber will be the main key to the proof of lemma \ref{T3}. Sample without replacement from the faces, apply the faces to $C_0$ in the reverse order and let $T_3$ denote once more the first time that the product of the faces picked is a chamber. Let $a$ be the number of chambers, then,
$$\pi(C)=\sum_{l} \sum_{F_i \neq F_j} \prcond{F_lF_{l-1}\ldots F_1C_0=C}{T_3=l} \pr{T_3=l}=  \frac{1}{a}$$
which is true because
$$\prcond{F_lF_{l-1}\ldots F_1=C}{T=l}=$$ $$ \frac{\sum_{\substack{F_{i_1}F_{i_2} \ldots F_{i_l} =C\\ F_{i_1}F_{i_2} \ldots F_{i_{l-1}} \notin \mathcal{C} \\ F_i \neq F_j}} \pr{ \substack{ F_{i_1},F_{i_2}, \ldots, F_{i_l} \\ \mbox{ are picked when sampling}\\ \mbox{ without replacement } \\ l \mbox{ times}} }}{\sum_{D \in \mathcal{C}}  \sum_{\substack{F_{i_1}F_{i_2} \ldots F_{i_l} =D\\ F_{i_1}F_{i_2} \ldots F_{i_{l-1}} \notin \mathcal{C} \\ F_i \neq F_j} }\pr{ \substack{ F_{i_1},F_{i_2}, \ldots, F_{i_l} \\ \mbox{ are picked when sampling}\\ \mbox{ without replacement } \\ l \mbox{ times}} }}= \frac{1}{a}$$
because of the symmetry conditions.
\end{proof}

\begin{lemma}\label{T3}
If the symmetry conditions hold then $T_3$ is a strong stationary time.
\end{lemma}

\begin{proof}
Let $a$ be the number of chambers of $\mathcal{A}$. To prove that 
$$\prcond{C^t=C}{T_3=t}= \pi(C) $$
consider at first the case $t=1$ and remember that $w(c)$ denotes the weight of a chamber $C$ when viewed as a face.
$$\prcond{C^1=C}{T_3=1}= \frac{w(C)}{\sum_{D \in \mathcal{C}} w(D)} = \frac{1}{a}$$
where $a$ is the number of chambers. But then because of the symmetry condition:
$$\prcond{C^2=C}{T_3\leq 2}=\frac{\sum_{F_{i_1}F_{i_2}=C } w(F_{i_1}) w(F_{i_2})}{\sum_{D \in \mathcal{C} }  \sum_{ F_{i_1}F_{i_2}=D} w(F_{i_1}) w(F_{i_2})}= \frac{1}{a} $$
and inductively, because of the weight invariant action of $G$. 
$$\prcond{C^t=C}{T_3=t} = \frac{1}{a} $$
\end{proof}

\begin{remark}
The strong stationary time of this section is not a strong stationary time if the symmetry conditions do not hold. To see this,
$$\prcond{C^1=C}{T_3=1}= \frac{w(C)}{\sum_{D \in \mathcal{C}} w(D)} $$
which is not necessarily equal to $\pi(C)$. 
\end{remark}

\section{Proof of Theorem \ref{sym.con}}
\begin{proof}
In this section, we assume that a group $G$ acts on $V$ preserving the hyperplane arrangement $\mathcal{A}$ so that the action restricted on the chambers is transitive. Assume the symmetry conditions described by equation \ref{condition}.

In this case we saw that the first time that the product of the faces picked is a strong stationary time. To bound the separation distance consider another union bound. 
$$P(T_3>t)\leq \sum^m_{i=1}\left(1 - \sum_{\substack{F \in \mathcal{F}\\ \sigma(F) \neq 0}}w(F) \right)^t$$
\end{proof}

\section{Acknowledgements}
I would like to thank  Persi Diaconis for all the comments  and suggestions and all the discussions  that we had concerning this work.

 \bibliographystyle{plain}
\bibliography{hyperplane_arrangements}

\end{document}